\documentclass{amsart} 
 \usepackage{amsmath}
\usepackage{amssymb}
\usepackage{amsfonts}
      
  \usepackage{a4wide}

\newtheorem{theorem}{Theorem}[section]
\newtheorem{lem}[theorem]{Lemma}
\newtheorem{cor}[theorem]{Corollary}
\newtheorem{prop}[theorem]{Proposition}

\theoremstyle{definition}

\newtheorem{example}[theorem]{Example}

\newtheorem{remark}[theorem]{\it Remark}
\newtheorem{remarks}[theorem]{\it Remarks}

\numberwithin{equation}{section}

\def \a{{\alpha}}
\def \b{{\beta}}

\def \e{{\varepsilon}}

\def \l{{\lambda}}

\def \p{{\varphi}}

\def \m{{\mu}}

\def \L{{\Lambda}}

\def \S{{\mathcal S}}

\def \C{{\cal C}}

\def \E{{\bf E}\, }

\def \N{{\bf N}}

\def \P{{\bf P}}

\def \qq{{\qquad}}
\def \R{{\bf R}}

\def \T{{\bf T}}

\def \Z{{\bf Z}}

\def \dd{{\rm d}}

\def \noi{{\noindent}}

\def\E{{\mathbb E \,}}

\def \T{{\mathbb T}}

\def\P{{\mathbb P}}

\def\R{{\mathbb R}}

\def \F{{\mathcal F}}
 
\def\Z{{\mathbb Z}}

\def\N{{\mathbb N}}

\def\C{{\mathbb C}}  
  
\scrollmode

  \title[Convergence of Almost periodic series]{Pointwise convergence of   almost periodic Fourier series and associated series of dilates}

\begin{document}  
 \author{Christophe Cuny}
\address{Universit\'e de la Nouvelle-Cal\'edonie
Equipe ERIM,
B.P. 4477,
F-98847 Noumea Cedex}.
\email{cuny@univ-nc.nc}

\author{Michel Weber}

\address{IRMA, 10 rue du G\'en\'eral Zimmer,
67084 Strasbourg Cedex, France}
\email{michel.weber@math.unistra.fr}


\maketitle 
 \begin{abstract}  Let  $\S^2$ be  the Stepanov space
 with norm  $\|f\|_{\S^2}=\sup_{x\in \R}\big( \int_x^{x+1} |f(t)|^2\dd t\big)^{1/2}$.  
Let $ \lambda_n\uparrow\infty$.
Let $(a_n)_{n\ge 1}$  be satisfying Wiener's condition:  
 $
\sum_{n\ge 1} \big(\sum_{k\, :\, n\le \lambda_k \le n+1}|a_k|\big)^2 <\infty$.   We establish the following maximal inequality 
\begin{equation*} 
\big\| \sup_{N\ge 1} \big|\sum_{n=1}^Na_n{\rm e}^{i\lambda_n t}\big|
\, \big\|_{\S^2}\le C\, \Big( \sum_{n\ge 1} \big(\sum_{k\, :\, n\le \lambda_k \le n+1}|a_k|\big)^2 \Big)^{1/2}
\end{equation*}
 where  $C>0$  is  a universal constant. 
  Moreover, the series $\sum_{n\ge 1} a_n{\rm e}^{it\lambda_n }$ converges 
for $\lambda$-a.e. $t\in \R$.  We give a simple and direct proof. This contains  as a  special case, 
  Hedenmalm and Saksman result for Dirichlet series. We also obtain maximal inequalities for corresponding series of dilates.   
Let $(\lambda_n)_{n\ge 1}$, $(\mu_n)_{n\ge 1}$ be  non-decreasing sequences of  real numbers 
greater than 1.  We prove the following interpolation theorem. Let $1\le p,q\le 2$ be such that $1/p+1/q=3/2$. 
 There exists $C>0$ such that for any sequence  $(\alpha_n)_{n\ge 1}$ and  
$(\beta_n)_{n\ge 1}$ of complex numbers such that 
 $  \sum_{n\ge 1} \big(\sum_{k\,:\, n\le \lambda_k< n+1}|\alpha_k|\,\big)^p
<\infty$ and $\sum_{n\ge 1} \big(\sum_{k\,:\, n\le \mu_k< n+1} |\beta_k|\,\big)^q
<\infty$,
  we have 
$$
\Big\|\sup_{N\ge 1} \big|\sum_{n=1}^N \alpha_n D(\lambda_n t)\big|\, \Big\|_{\S^2}
\le C \Big(\sum_{n\ge 1} \big(\sum_{k\,:\, n\le \lambda_k< n+1}|\alpha_k|\big)^p \Big)^{1/p}\Big(\sum_{n\ge 1} \big(\sum_{k\,:\, n\le
 \mu_k< n+1}|\beta_k|\big)^q\Big)^{1/q } 
$$
where $D(t)= \sum_{n\ge 1}\beta_n {\rm e}^{i\mu_n t}$ is defined in $\S^2$. 
 Moreover, the series $\sum_{n\ge 1} 
\alpha_n D(\lambda_nt)$ converges in $\S^2$ and for $\lambda$-a.e. $t\in \R$. We further show that if  $\{\l_k, k\ge 1\}$ satisfies the  following condition
\begin{eqnarray*} \sum_{  k\not=\ell\,,\,  k'\not=\ell'\atop (k,\ell)\neq(k',\ell')  
  }
\big(1-|(\l_k-\l_\ell)-(\l_{k'}-\l_{\ell'}) |\big)_+^2 \, <\infty , \end{eqnarray*}
then the series $\sum_{k} a_k {\rm e}^{i\l_kt}$ converges on a set of positive Lebesgue measure,  only if  the series  $\sum_{k=1}^\infty |a_k|^2$
converges. The above condition is in particular fulfilled when $\{\l_k, k\ge 1\}$ is a Sidon sequence.
  
 \end{abstract}
 

 \section{Introduction.}\label{s1}     
  We study almost everywhere convergence properties of   almost periodic Fourier series  in the Stepanov space $\S^2$ and of corresponding series of dilates. This space is   defined as the sub-space of functions $f$ of $L^2_{\rm loc}(\R)$     verifying the   following analogue of Bohr almost periodicity property: {\it For all $\e>0$, there exists $K_\e>0$ such that for any $x_0\in \R$, there exists $\tau\in [x_0, x_0+K_\e]$ such that  $\| f(.+\tau)-f(.)\|_{\S^2} \le \e$}. The Stepanov norm in  $\S^2 $ is defined by 
  $$\|f\|_{\S^2}=\sup_{x\in \R}\Big( \int_x^{x+1} |f(t)|^2\dd t\Big)^{1/2}. $$

Recall some basic facts.   By the fundamental theorem on almost periodic functions see \cite[p.\,88]{Be}, the Stepanov space
  $\S^2$ coincides with the 
closure of   the set of generalized trigonometric polynomials 
$\{\sum_{k=1}^n a_k {\rm e}^{i\lambda_k t}\,:\, \alpha_k \in \C,\, 
\lambda_k \in \R\}$ with respect to this norm. 
 It is clear by considering for instance $f= \chi_{ [0,1]}$ that
the space $\{f\in L^2_{\rm loc}(\R)\,:\, \|f\|_{\S^2}<\infty\}$ is strictly larger than 
$\S^2$. 
Introduce also the Besicovitch  semi-norm of order 2 of $f\in L^2_{\rm loc}(\R)$ 
 \begin{equation}\label{Besicovitch}
\|f\|_{\mathcal B^2}=\limsup_{T\to \infty}\Big( \frac{1}{2T} \int_{-T}^T |f(t)|^2
\dd t\Big)^{1/2}.\end{equation}
 
  For every $\lambda\in \R$ and every  $f\in L^1_{\rm loc}(\R)$ define the 
  Fourier coefficient $\hat f(\l)$ of exponent $\l$ of $f$ by
\begin{equation}\label{coeff}
\widehat f(\l)=\lim_{T\to \infty} \frac{1}{2T} \int_{-T}^{T} f(x) {\rm e}^{-i\l x } \dd x\, , 
\end{equation}
whenever the limit exists. 
 It is easily seen, by approximating by generalized trigonometric polynomials in the Stepanov norm, that the above limit  exists for every $f\in \S^2$ and every $\lambda\in \R$. Moreover, 
for any finite family $\lambda_1,\ldots ,\lambda_n\in \R$,  we have by Parseval equation in $\mathcal B^2$ (\cite[p.\,109]{B}),
$$\sum_{k=1}^n |\hat f(\l_k)|^2\le \|f\|_{\mathcal B^2}^2\le \|f\|_{\S^2}^2 .  $$

\smallskip

In particular, for $f\in \S^2$,   $\L:=\{\lambda \in \R\, :\, \hat f(\lambda)
\neq 0\}$ is countable. We shall call $\L$ the   (set of) Fourier exponents of $f$. 
Let $f\in \S^2$ with set of Fourier exponents $\L$. We have
\begin{equation}\label{Bessel}
\sum_{\l\in \L} |\hat f(\l)|^2\le \|f\|_{\mathcal B^2}^2\le \|f\|_{\S^2}^2\, .
\end{equation}

We then define formally the Fourier series of 
$f\in \S^2$ as 

$$
\sum_{\l\in \L}\hat f(\l)e^{i\lambda\cdot}\, .
$$
 Notice that the set $\L\cap[-A,A]$ may be infinite for a given $A>0$.  

\vskip 2 pt

In this paper we are   interested in  the convergence of the Fourier series of $f$ (to $f$) either in the 
Stepanov sense or in the almost everywhere sense,  and the same sort of consideration will motivate us in the study of associated series of dilates. This second question is actually our main objective. See Section \ref{s3}.
\vskip 3pt
  Concerning convergence of the Fourier series, it is necessary to recall    Bredihina's  extension to  $\S^2$ of Kolmogorov's theorem asserting that if $s_n(x)$ are the partial sums of the Fourier series of a function $f\in L^2(\T)$, then $s_{m_n}(x)$ converges almost everywhere to $f$ provided that $m_{n+1}/m_n \ge q >1$.     
  Bredihina showed in  \cite{Bredihina}    that the Fourier series of a function in $\S^2$ with  $\alpha$-separated frequencies ($\alpha>0$), namely $|\lambda_k-\lambda_\ell|\ge \a>0$ for all $k,\ell$, $k\not=\ell$,  converges almost everywhere along any 
exponentially increasing subsequence. That is, for every 
$\rho>1$, the sequence $\{\sum_{1\le k\le \rho^n} \hat f(\lambda_k){\rm e}
^{i\lambda_k t}, n\ge 1\}$ converges for $\lambda$-almost every $t\in \R$. 
 The corresponding maximal inequality has been recently obtained by 
Bailey \cite{Ba} who also considered Stepanov spaces of higher order.

  \begin{remark} \label{marcin}   For a  short proof of Kolmogorov's Theorem see Marcinkiewicz \cite{M},  who showed that this follows from  Fejer's Theorem (\cite[Th. 3.4-(III)] {Z}) and  the classical fact that if a series $\sum u_n$ with partial sums $s_n$ has infinitely many  lacunary gaps and is summable $(C,1)$ to sum $s$, then $s_n\to s$.  See Theorem 1.27 in Chapter III of \cite{Z}.
\end{remark}


In view of Carleson's theorem, a natural question is whether the "full" series converges for any $f\in \S^2$. 
 
 That question has been addressed in the very specific situation 
of Dirichlet series by Hedenmalm and Saksman \cite{HS}. A simplified proof may be found in Konyagin and Queff\'elec \cite{KQ} (see also below). 
They proved the following. Let  $\l$ denote here and throughout the Lebesgue measure on the real line.

\begin{theorem}
Let $(a_n)_{n\ge 1}$ be complex numbers such that $\sum_{n\ge 1}n|a_n|^2
<\infty$. Then the series $\sum_{n\ge 1}a_nn^{it}$ converges   
$\lambda$-almost everywhere. 
\end{theorem}
Their condition is optimal when $(a_n)_{n\ge 1}$ is non-increasing. 
However, if $(a_n)_{n\ge 1}$ is supported say on $\{2^n\, :\, n\ge N\}$ 
the corresponding series is a standard (periodic) trigonometric 
series and in that case, the optimality is lost, since the condition is much stronger than Carleson's condition. 

On the other hand, it follows from Wiener \cite{Wi}  that the series $\sum_{n\ge 1}a_nn^{it}$ 
converges in $\S^2$ provided that 
\begin{equation}\label{weak-cond}
 \sum_{n\ge 0}  \Big(\sum_{k=2^n}^{2^{n+1}-1} |a_k|\Big)^2 <\infty\, .
 \end{equation}
More precisely, the sequence of partial sums converges in $ \S^2$ to a limit $f\in \S^2$. If $a_n>0$ for every $n$,   the converse is also true, see Tornehave \cite{To}.
  \vskip 2 pt
 Our first goal (see the next section) is to prove that \eqref{weak-cond} is sufficient for 
 $\lambda$-a.e. convergence and to provide the corresponding maximal inequality. Moreover, it will turn out 
 that the problem of the $\lambda$-almost everywhere convergence 
 of series $\sum_{n\ge 1} a_n {\rm e}^{i\lambda_n t}$ can be reduced to 
 the study of Dirichlet series. 
 
 \vskip 1 pt In doing so, we obtain a Carleson-type theorem for almost periodic series and make the link with the study of 
 almost everywhere convergence of the Fourier series associated with 
 Stepanov's almost periodic functions.
 \vskip 2 pt

Then, in Section \ref{s3}, we consider associated series of dilates and obtain a sufficient condition for almost everywhere convergence. We further prove 
an  interpolation theorem. 
Finally, in Section \ref{s5}, we obtain a general necessary condition for the convergence almost everywhere of series of functions. The condition involves correlations of order $4$.   
 As an application, we show for instance that if $\{\l_k, k\ge 1\}$ is a Sidon sequence, and the series $\sum_{k} a_k {\rm e}^{i\l_kt}$ converges on a set of positive $\l$-measure, then the
series  $\sum_{k=1}^\infty |a_k|^2$ converges.

  

\section{Almost everywhere convergence of almost periodic Fourier 
series}\label{s2}   
 
 We start with the proof by Konyagin and Queff\'elec of 
 Hedenmalm and Saksman's result, to which we add a maximal inequality.

\begin{prop}\label{prop-KQ}
There exists $C>0$ such that for any sequence  $(a_n)_{n\ge 1}$ of complex numbers such that $\sum_{n\ge 1}n|a_n|^2
<\infty$, 
\begin{equation}\label{inemax-diri}
\Big\| \sup_{n\ge 1}|\sum_{k=1}^n a_k k^{i\cdot }|\, \Big\|_{\S^2}\le C \,
\Big(\sum_{n\ge 1}n|a_n|^2\Big)^{1/2}\, .
\end{equation}
\end{prop}

Before giving the proof, it is necessary to recall some classical but important facts. Let $g\in L^p(\T)$, $1<p<\infty$. Consider the maximal operator 
$$T^* g(x) = \sup_{L=0}^\infty \Big|\sum_{|k|\le L}\widehat{g}(k) {\rm e}^{2i\pi k x}\Big| . $$
For $f\in L^p(\R)$ consider analogously the maximal operator
$$C^* f(x) = \sup_{T>0}  \Big|\int_{-T}^T \widehat{f}(t) {\rm e}^{i xt}\dd t\Big| . $$
An operator $U$ on $L^p$ is  said  strong $(p,p)$  if $\|Uf\|_p\le C_\p \|f\|p$ for all $f\in L^p$. The fact that strong $(p,p)$, $1<p<\infty$, for $T^*$ is equivalent to strong $(p,p)$ for $C^*$ follows from known elementary arguments (see \cite[p.\,166]{AC}). We refer to \cite[Theorem 1]{Hu} concerning the deep fact that $T^*$ is strong  $(p,p)$, $1<p<\infty$ and we shall call it "the Carleson-Hunt theorem" when $p=2$. We will freely use the fact the $C^*$ is consequently strong  $(p,p)$, $1<p<\infty$.
\begin{proof} We first notice that it is enough to prove that 
\begin{equation}\label{est}
\Big\| \sup_{n\ge 1}|\sum_{k=1}^n a_k k^{i\cdot }|\, \Big\|_{L^2[0,1]}\le C 
\Big(\sum_{n\ge 1}n|a_n|^2\Big)^{1/2}\, .
\end{equation}
Indeed, then the desired result follows from the fact that 
$$
\sum_{k=1}^n a_k k^{i(t+x)}= \sum_{k=1}^n (a_kk^{ix}) k^{it}\, ,
$$
since we may apply the above estimate to the sequence 
$(a_nn^{ix})_{n\ge 1}$ whose modulus are the same as the ones of the sequence $(a_n)_{n\ge 1}$.

Let us prove \eqref{est}. Define $h\in L^2(\R)$ by setting 
$h\equiv 0$ on $(-\infty, 1)$ and for every  $n\in \N$, 
$h(x)=a_n$ whenever $x\in [n,n+1)$. 

\smallskip

Let $N\ge 1$. We have 

\begin{gather*}
\sum_{n=1}^N a_n n^{it}= \sum_{n=1}^N a_n \int_n^{n+1}\big( 
{\rm e}^{it \log n}-{\rm e}^{it\log x}\big) dx + \int_1
^{N+1}h(x) {\rm e}^{it\log x}dx\\
=\sum_{n=1}^N a_n \int_n^{n+1}\big( 
{\rm e}^{it \log n}-{\rm e}^{it\log x}\big) dx + \int_0^{\log (N+1)} {\rm e}^x
h({\rm e}^x) {\rm e}^{itx}
dx\, .
\end{gather*}
Now, for every $x\in [n,n+1)$, 
$$
\big|{\rm e}^{it \log n}-{\rm e}^{it\log x}\big|\le \frac{t}{n}\, .
$$
Hence, 
$$
\sum_{n\ge1}\Big| a_n \int_n^{n+1}\big( 
{\rm e}^{it \log n}-{\rm e}^{it\log x}\big) dx\Big|
\le t \Big(\sum_{n\ge 1}n|a_n|^2\Big)^{1/2} \Big(\sum_{n\ge 1}\frac{1}{n^3}\Big)^{1/2}
\, .
$$

On another hand, $\int_0^{+\infty} {\rm e}^{2x}|h|^2({\rm e}^x)dx 
= \int_1^{+\infty} u|h|^2(u)du \le  \sum_{n\ge 1}(n+1)|a_n|^2
<\infty$. Hence, since $C^*$ is strong $(2-2)$, 
$$\Big\|\sup_{N\ge 1}\big|  \int_0^{\log (N+1)} {\rm e}^xh({\rm e}^x) {\rm e}^{itx}
dx\big|\, \Big\|^2_{2,dt}\le C \int_0^{+\infty} {\rm e}^{2x}|h|^2({\rm e}^x)dx \, .
$$
Hence \eqref{inemax-diri} follows. \end{proof}


 
 We now derive an improved version of Proposition \ref{prop-KQ}.

 \begin{theorem}\label{prop-dirichlet}
There exists $C>0$, such that for every sequence $(a_n)_{n\ge 1}$  of complex numbers satisfying 
 \eqref{weak-cond},
 \begin{equation}\label{inemax-diri-2}
\Big\| \sup_{n\ge 1}|\sum_{k=1}^n a_k k^{i\cdot }|\, \Big\|_{\S^2}\le C 
\Big(\sum_{n\ge 0}(\sum_{k=2^n}^{2^{n+1}-1}|a_k|)^2\Big)^{1/2}\, .
\end{equation}  
Moreover,  $\sum_{n\ge 1}a_n n^{it}$ converges for $\lambda$-a.e. $t\in \R$. 
 \end{theorem}
\begin{remarks} \label{2.3} The proof of Theorem \ref{prop-dirichlet} 
makes use of Carleson-Hunt's theorem ($T^*$ is strong ($2-2$)) and of Proposition 
\ref{prop-KQ}. The latter was proved  using that $C^*$ is strong ($2-2$), which is  equivalent 
to Carleson-Hunt's theorem. On the other hand, 
given any sequence $(b_n)_{n\ge 1}\in \ell^2$, applying Theorem \ref{prop-dirichlet} with $(a_n)_{n\ge 1}$ such that $a_{2^k}=b_k$ and 
$a_n=0$ otherwise, we see that Theorem \ref{prop-dirichlet} implies 
Carleson-Hunt's theorem, hence is equivalent to it. 
We shall see below that Theorem \ref{prop-dirichlet} allows to 
treat almost everywhere convergence of series $\sum_{n\ge 1} 
b_n{\rm e}^{it\lambda_n }$ for non-decreasing sequences $(\lambda_n)_{n\ge 1}$. Notice that Theorem \ref{prop-dirichlet} 
corresponds to the case where $\lambda_n=\log n$. 
 For more on Carleson-Hunt's theorem we refer to Lacey \cite{La1}. See also J\o rsboe and Mejlbro \cite{JM}.
\end{remarks}
\begin{proof} As in the previous proof, it is enough to 
  prove a maximal inequality in $L^2([0,1])$. 
 We shall first work along the  subsequence $(2^n-1)_{n\ge 1}$. 
 \vskip 2pt 
 Let $n\ge 1$ and define $S_{k,n}:=\sum_{\ell=2^n }^k a_k$ for every 
$2^n \le k\le 2^{n+1}-1$ and $S_{2^n-1,n}=0$. In particular, 
for every $2^n \le k\le 2^{n+1}-1$, 
$$
|S_{k,n}|\le \sum_{j=2^n}^{2^{n+1}-1} |a_j|\, ,
$$
a fact that should be used freely in the sequel.

\smallskip
By Abel summation by part, 
we have 
 \begin{gather*}
\sum_{k=2^n}^{2^{n+1}-1} a_kk^{it} =\sum_{k=2^n}^{2^{n+1}-1} (S_{k,n}-S_{k-1,n})k^{it} = \sum_{k=2^n}^{2^{n+1}-1} S_{k,n}(k^{it}-(k+1)^{it}) 
+ 2^{(n+1)it}S_{2^{n+1}-1,n}\, .
\end{gather*}
 Since $2^{(n+1)it}={\rm e}^{i(n+1)t\log 2}$ and by our assumption 
$\sum_{n\ge 1} |S_{2^{n+1}-1,n}|^2 <\infty$, it follows from Carleson's theorem that 
$$
\Big\|\sup_{N\ge 1} \sum_{n=1}^N  S_{2^{n+1}-1,n}2^{(n+1)it} \Big\|_{L^2([0,1],dt)}
\le C \Big( \sum_{n\ge 1} |S_{2^{n+1}-1,n}|^2\Big)^{1/2}\, .
$$


\smallskip

Hence, we are back to control the $L^2$-norm of 
$$\sup_{N\ge 1}\Big|\sum_{n= 1}^N \sum_{k=2^n}^{2^{n+1}-1} S_{k,n}(k^{it}-(k+1)^{it})\Big|\, .
$$
 But  we have, 
\begin{eqnarray*}
k^{it}-(k+1)^{it }&=&{\rm e}^{it\log k}-{\rm e}^{it\log(k+1)}\cr &=&
{\rm e}^{it\log k}(1-{\rm e}^{it\log (1+1/k)}+\frac{it}{k})
- \frac{it}{k}{\rm e}^{it\log k}\ = \ u_k(t)-\frac{it}{k}{\rm e}^{it\log k}\, .
\end{eqnarray*}

Now there exists $C>0$ such that $|u_k(t)|\le \frac{C(t+t^2)}{k^2}$. 
Hence, 
\begin{eqnarray*}
\sum_{n\ge 1} \sum_{k=2^n}^{2^{n+1}-1} |S_{k,n}|\, |u_k(t)|
&\le &C (t+t^2)\sum_{n\ge 1}\frac{\sum_{k=2^n}^{2^{n+1}-1}|a_k|}{2^n}\\
\cr &\le &C(t+t^2)\Big(\sum_{n\ge 1}\big(\sum_{k=2^n}^{2^{n+1}-1}|a_k|
\big)^2\Big)^{1/2}\, .
\end{eqnarray*}

\smallskip

It remains to control
 $$
\sup_{N\ge 1}\Big|\sum_{n= 1}^N \sum_{k=2^n}^{2^{n+1}-1} \frac{S_{k,n}}{k} {\rm e}^{it\log k}\Big|
\, .
$$
 But we are exactly in the situation of Proposition \ref{prop-KQ}. 
Hence  
 \begin{eqnarray*}
\Big\| \sup_{N\ge 1}\Big|\sum_{n= 1}^N \sum_{k=2^n}^{2^{n+1}-1} \frac{S_{k,n}}{k} {\rm e}^{it\log k}\Big|\Big\|_{L^2([0,1],dt)}&\le C &\Big(\sum_{n\ge 1}
\sum_{k=2^n}^{2^{n+1}-1} k\frac{|S_{k,n}|^2}{k^2} \Big)^{1/2} 
\cr  &\le &\Big(\sum_{n\ge 1}\big(\sum_{k=2^n}^{2^{n+1}-1}|a_k|
\big)^2\Big)^{1/2}\  < \ \infty  .
\end{eqnarray*}
Let $n\ge 1$ and $2^n\le \ell \le 2^{n+1}-1$. We have 
$$
\Big|\sum_{k=1}^{\ell}a_n k^{it} -\sum_{k=1}^{2^n-1}a_n k^{it}\Big|
\le \sum_{k=2^n}^{2^{n+1}-1}|a_k|\, .
$$

Hence, 

$$
\sup_{N\ge 1} \Big|\sum_{n=1}^N a_n{\rm e}^{it\log n}\Big|\le 
\sup_{N\ge 1} \Big|\sum_{n=1}^{2^{N}-1} a_n{\rm e}^{it\log n}\Big|
+\Big(\sum_{n\ge 1}\big(\sum_{k=2^n}^{2^{n+1}-1}|a_k|
\big)^2\Big)^{1/2}\, .
$$
So, \eqref{inemax-diri-2} is proved. The $\lambda$-almost everywhere convergence may be proved by a standard procedure thanks 
to the maximal inequality. Alternatively, following all the steps of the proof of the maximal inequality allows to give a more direct proof. 
\end{proof}

\vskip 4 pt
As a corollary we deduce
 \begin{theorem}\label{gen-prop}
Let $(\lambda_n)_{n\ge 1}$ be an increasing sequence of positive real 
numbers tending to $\infty$.
Let $(a_n)_{n\ge 1}$ be such that 
\begin{equation}\label{Wiener-cond}
\sum_{n\ge 1} \Big(\sum_{k\, :\, n\le \lambda_k \le n+1}|a_k|\Big)^2 <\infty\, .
\end{equation}
There exists a universal constant $C>0$ such that 
\begin{equation}\label{inemax-gen}
\Big\| \sup_{N\ge 1} \big|\sum_{n=1}^Na_n{\rm e}^{i\lambda_n t}\big|
\, \Big\|_{\S^2}\le C \Big(\sum_{n\ge 1} \Big(\sum_{k\, :\, n\le \lambda_k \le n+1}|a_k|\Big)^2\Big)^{1/2}
\end{equation}
Moreover, the series $\sum_{n\ge 1} a_n{\rm e}^{it\lambda_n }$ converges 
for $\lambda$-a.e. $t\in \R$.
\end{theorem}

\begin{proof} Write $u_n:=[2^{\lambda_n}]$. Hence $(u_n)_{n\ge 1}$ is a
non-decreasing sequence of integers.  That sequence may overlap 
from time to time. So let $(v_k)_{k\ge 1}$ be a strictly increasing 
sequence of integers with same range as $(u_n)_{n\ge 1}$.

Define a sequence $(b_n)_{n\ge 1}$ as follows. Let $n\ge 1$ be 
such that there exists $k\ge 1$ such that $n=v_k$. Then, set 
 $b_n:=\sum_{\ell \, :\, u_\ell =v_k} a_\ell$. If there is no 
 $k\ge 1$ such that $n=v_k$, set $b_n:=0$.

\smallskip

We first control  
$$
\sup_{N\ge 1}\Big|\sum_{n =1}^N  b_n {\rm e}^{it\log_2 n}\Big|\, ,$$ 
where $\log_2$ stands for the logarithm in base 2.  


\smallskip

By Theorem \ref{prop-dirichlet}, we have 
 $$
\Big\| \sup_{N\ge 1}|\sum_{n =1}^N  b_n {\rm e}^{i\log_2 n
\cdot }|\Big\|_{\S^2}^2\le C \sum_{n\ge 0}  \Big(\sum_{k=2^n}^{2^{n+1}-1} |b_k|\Big)^2 =\sum_{n\ge 0} \Big(\sum_{\ell\, :\, 2^n \le u_\ell \le 2^{n+1}-1} 
|b_\ell|\Big)^2 \, . 
$$
Now, if $2^n \le u_\ell \le 2^{n+1}-1$, then $n\le\lambda_{\ell} 
\le n+1$ and our first step is proved. 
\smallskip

Let $q\ge p$ be integers. There exist integers $q'\ge p'$ such 
that $v_{p'}=u_p$ and $v_{q'}=u_q$. We have 
\begin{gather*}
\Big|\sum_{k=p}^q a_k{\rm e}^{it\lambda_k}- \sum_{k=v_{p'}}^{v_{q'}} b_k
{\rm e}^{it \log_2 u_k} \Big|\\
\le \sum_{k\, :\, u_k=u_p}|a_k| 
+ \sum_{k\, :\, u_k=u_q}|a_k|  + 
\sum_{\ell=p'}^{q'} \sum_{k\, :\,  u_k= v_\ell} 
|a_k|\, |{\rm e}^{it \lambda_k}-{\rm e}^{it \log_2 u_k}| 
\end{gather*}
Clearly, it suffices to control

$$
\sum_{n\ge 0} \quad \sum_{\ell\, :\, 2^n \le v_\ell \le 2^{n+1}-1} 
\quad \sum_{k\, :\, u_k=v_\ell} |a_k|\, |{\rm e}^{it \lambda_k}-{\rm e}^{it \log_2 u_k}| \, .
$$
Now, for $2^n \le v_\ell \le 2^{n+1}-1$ and $u_k=v_\ell$, 
using that $u_k\le 2^{\lambda_k}\le u_k +1$, we see that 
$|\log_2 (2^{\lambda_k})-\log_2u_k|\le \frac{C}{u_k}$ and that 

\begin{gather*}
|{\rm e}^{it \lambda_k}-{\rm e}^{it \log_2 u_k}| 
=|{\rm e}^{it \log_2(2^{\lambda_k})}-{\rm e}^{it \log_2 u_k}|
\le \frac{C|t|}{u_k}\le \frac{C|t|}{2^n}\, . 
\end{gather*}

Hence, using Cauchy-Schwarz,

\begin{eqnarray*}
\sum_{n\ge 0} \quad \sum_{\ell\, :\, 2^n \le v_\ell \le 2^{n+1}-1} 
\quad \sum_{k\, :\, u_k=v_\ell} |a_k|\, |{\rm e}^{it \lambda_k}-{\rm e}^{it \log u_k}| &\le  &
Ct\sum_{n\ge 0}2^{-n} \sum_{k\, :\, 2^n \le u_k\le 2^{n+1}-1}|a_k|\\
\cr &\le &Ct \Big( \sum_{n\ge 0}\big(\sum_{k\, :\, 2^n \le u_k\le 2^{n+1}-1}
|a_k|\big)^2\Big)^{1/2}\, ,
\end{eqnarray*}
which converges by our assumption. \end{proof}

We shall now derive an almost everywhere  convergence result 
concerning the Fourier series of an almost periodic function in $\S^2$. 
We shall first recall known results about  norm convergence. 

\medskip

Let $(\lambda_n)_{n\ge 1}$ be a (non-necessarily increasing) 
of positive real numbers. As already mentionned (in the case of Dirichlet series), by Wiener \cite{Wi}, see also Tornehave \cite{To} 
if 
\begin{equation}\label{gen-wiener}
\sum_{n\ge 0} \Big(\sum_{k\ge 1\, : n\le \lambda_k<n+1}|a_k|\Big)^2 <\infty
\, ,
\end{equation}
then $\sum_{n\ge 1}a_n{\rm e}^{i\lambda_n t}$ 
is the Fourier series of an element of $f\in \S^2$. 

On the other hand if $f\in \S^2$ admits a sequence of positive real numbers 
$(\lambda_n)_{n\ge 1}$ as frequencies and such that 
$\hat f(\lambda_n)\ge 0$ for every $n\ge 1$, then (Tornehave \cite{To})

 $$
\sum_{n\ge 0} \Big(\sum_{k\ge 1\, : n\le \lambda_k<n+1}|\hat f(\lambda_k)|\Big)^2
\le C \|f\|_{\S^2}^2\, .
$$
Hence, \eqref{gen-wiener} holds. 

\medskip

Condition \eqref{gen-wiener} is thus optimal for deciding whether 
$\sum_{n\ge 1}a_n{\rm e}^{i\lambda_n t}$ is the Fourier series of an element of $\S^2$ or not. One can not however expect that it is always necessary, so we should provide a counterexample in Proposition \ref{counterexample} below. 

\medskip



\medskip

Let $f\in \S^2$ be such that $\Lambda\subset [0,+\infty)$ (that restriction 
may be obviously removed). Assume that $\L$ is $\alpha$-separated for some $\alpha>0$ and write $\L:=\{\lambda_1<\lambda_2\ldots\}$.  Then, 
$$
\frac{\alpha}{C}\sum_{n\ge 0} (\sum_{k\ge 1\, : n\le \lambda_k<n+1}|\hat f(\l_k)|)^2 \le \sum_{n\ge 1}|\hat f(\l_n)|^2\le \|f\|_{\S^2}^2\le C \sum_{n\ge 0} (\sum_{k\ge 1\, : n\le \lambda_k<n+1}|\hat f(\l_k)|)^2\, .
$$

In particular, we have the following 
direct consequence of Theorem \ref{prop-dirichlet}.

\begin{cor}
Let $f\in \S^2$ be such that $\Lambda\subset [0,+\infty)$. Assume that $\L$ is $\alpha$-separated for some $\alpha>0$. 
There exists $C>0$, independent of $f$ and $\alpha$ such that 
$$
\big\|\sup_{N\ge 1}\big|\sum_{n=1}^N \hat f(\l_n){\rm e}^{i\l_n\cdot}
\big|\, \big\|_{\S^2}\le C \frac{\|f\|_{\S^2}}{\alpha}\, .
$$
Moreover, the series $\sum_{n\ge 1} \hat f(\l_n){\rm e}^{i\l_n\cdot}$ converges 
for $\lambda$-almost every $t\in \R$.
\end{cor}



\medskip

We now give an example of Fourier series converging in $\S^2$ 
while \eqref{gen-wiener} does not hold. 
Let us first recall the following result of Halasz, see Queff\'elec \cite{Q}. 

\medskip

\begin{lem}
There exists $C>0$ such that for every iid Rademacher variables 
$(\varepsilon_n)_{n\ge 1}$  
\begin{equation}\label{HQ}
\E(\sup_{t\in \R} |\sum_{k=1}^n \varepsilon_kk^{it}|)
\le C \frac{n}{\log (n+1) }\, .
\end{equation}
\end{lem}

\medskip

\begin{prop}\label{counterexample}
Let $(\varepsilon_n)_{n\ge 1}$ be iid Rademacher variables on 
$(\Omega,\F,\P)$. For $\P$-almost all $\omega\in \Omega$, 
$\sum_{n\ge 1} \frac{\varepsilon_n(\omega)n^{it}}{n \sqrt{\log (n+1)}}$
converges in $\S^2$, while \eqref{Wiener-cond} is not satisfied 
(with $a_n=\frac{\varepsilon_n(\omega)}{n \sqrt{\log (n+1)}}$).
\end{prop}
\begin{proof} For every $n\ge 1$, every $2^n\le k\le 2^{n+1}$ and every $\omega\in \Omega$, we have 
 \begin{gather*}
\|\sum_{\ell=2^n }^k \frac{\varepsilon_\ell(\omega)\ell^{it}}
{\ell \sqrt{\log (\ell+1)}}\|_{\S^2} \le \sum_{\ell=2^n }^k \frac{1}
{\ell \sqrt{\log (\ell+1)}}
\le \frac{2}{\sqrt n}\underset{n\to +\infty}\longrightarrow 
0\, .
\end{gather*}

Hence, it suffices to prove that for $\P$-almost every 
$\omega\in \Omega$, $(\sum_{n=1}^{2^N}\frac{\varepsilon_n(\omega)n^{it}}{n \sqrt{\log (n+1)}})_{N\ge 1}$ converges in $\S^2$.

\medskip

Let $S_n(t):= \sum_{k=1}^n \varepsilon_kk^{it}$ ($S_0(t)=0$) and $u_n:=
(n\sqrt{\log(n+1)})^{-1}$. We have 

\begin{gather*}
\sum_{n=1}^{2^N}\frac{\varepsilon_n(\omega)n^{it}}{n \sqrt{\log (n+1)}} 
=\sum_{n=1}^{2^N}(S_n(t)-S_{n-1}(t))u_n= 
\sum_{n=1}^{2^N}S_n(t)(u_n-u_{n+1})+S_{2^N}(t)u_{2^N+1}\, .
\end{gather*} 

It follows from \eqref{HQ} that 

\begin{gather*}
\E\Big(\sum_{n\ge 1} \sup_{t\in \R}\big|S_n(t)(u_n-u_{n+1})\big|
\Big) <\infty\, ,\\
\E\Big(\sum_{n\ge 1} \sup_{t\in \R}\big|S_{2^N}(t)u_{2^N+1}
\big|\Big)<\infty\, ,
\end{gather*}
and the result follows. \end{proof}

\section{Convergence almost everywhere of associated series of dilates.}\label{s3}

\begin{theorem}\label{prop-series-1}
Let $(\lambda_n)_{n\ge 1}$ and $(\mu_n)_{n\ge 1}$ be  non-decreasing sequences of  real numbers 
greater than 1. Let $(\alpha_n)_{n\ge 1}$ be a sequence of complex numbers 
such that 
$$\sum_{n\ge 1} \Big(\sum_{k\,:\, n\le \lambda_k< n+1}|\alpha_k|\Big)^2
<\infty.$$ Let $(\beta_n)_{n\ge 1}\in \ell^1$. Then, $D(t):= \sum_{n\ge 1} \beta_n {\rm e}^{i\mu_nt}$ defines 
a continuous function on $\R$ (and in $S^2$) and there exists a universal constant $C>0$ such that 
\begin{equation}\label{inemax-series-1}
\Big\|\sup_{N\ge 1} \big|\sum_{n=1}^N \alpha_n D(\lambda_n t)\big|\, \Big\|_{\S^2}
\le C \Big(\sum_{n\ge 1}|\beta_n|\Big) \Big(\sum_{n\ge 1} \big(\sum_{k\,:\, n\le \lambda_k< n+1}|\alpha_k|\big)^2\Big)^{1/2 }\, .
\end{equation}
Moreover, the series $\sum_{n\ge 1} 
\alpha_n D(\lambda_nt)$ converges in $\S^2$ and for $\lambda$-a.e. $t\in \R$.
\end{theorem}

\begin{proof} Let $x\in \R$. The fact that $D$ is a continuous function in $S^2$ follows easily from the fact that $(\beta_n)_{n\ge 1}\in 
\ell^1$. We also have, for every $N\ge 1$,
$$
\Big|\sum_{n=1}^N \alpha_n D(\lambda_n t)\Big|\le \sum_{k\ge 1}|\beta_k| \Big|\sum_{n=1}^N \alpha_n{\rm e}^{it \lambda_n\mu_k}\Big|\, .
$$
By Theorem \ref{gen-prop}, we have 
\begin{gather*}
\int_x^{x+1} \sup_{N\ge 1} \Big|\sum_{n=1}^N \alpha_n{\rm e}^{it \lambda_n\mu_k}\Big|^2dt = \frac1{\mu_k}\int_{\mu_k x}^{\mu_k(x+1)} \sup_{N\ge 1}
\Big|\sum_{n=1}^N
\alpha_n{\rm e}^{it \lambda_n}\Big|^2dt\\
\le \frac{[\mu_k]+1}{\mu_k}\Big\| \sup_{N\ge 1} \big|\sum_{n=1}^N \alpha_n{\rm e}^{it \lambda_n}\big|\, \Big\|_{\S^2}^2\, ,
\end{gather*}
and \eqref{inemax-series-1} follows.

\smallskip

The convergence almost everywhere and in $\S^2$ follows by standard arguments. 
\end{proof}

We also have the following obvious corollary of Theorem 
\ref{gen-prop}, whose proof is left to the reader.

\begin{prop}\label{prop-series-2}
Let $(\lambda_n)_{n\ge 1}$ and $(\mu_n)_{n\ge 1}$ be  non-decreasing sequences of  real numbers 
greater than 1. Let $(\beta_n)_{n\ge 1}$ be a sequence of complex numbers 
such that 
$$\sum_{n\ge 1} \Big(\sum_{k\,:\, n\le \mu_k< n+1}|\beta_k|\Big)^2
<\infty . $$
Let $(\alpha_n)_{n\ge 1}\in \ell^1$. Then, $D(t):= \sum_{n\ge 1} \beta_n {\rm e}^{i\mu_nt}$ converges in $\S^2$ and  
 there exists a universal constant $C>0$ such that 
\begin{equation}\label{inemax-series-2}
\Big\|\sup_{N\ge 1} \big|\sum_{n=1}^N \alpha_n D(\lambda_n t)\big|\, \Big\|_{\S^2}
\le C \Big(\sum_{n\ge 1}|\alpha_n|\Big) \Big(\sum_{n\ge 1} \big(\sum_{k\,:\, n\le
 \mu_k< n+1}|\beta_k|\big)^2\Big)^{1/2 }\, .
\end{equation}
Moreover, the series $\sum_{n\ge 1} 
\alpha_n D(\lambda_nt)$ converges in $\S^2$ and for $\lambda$-a.e. $t\in \R$.
\end{prop}



\begin{theorem} \label{interpolation}
Let $(\lambda_n)_{n\ge 1}$ and $(\mu_n)_{n\ge 1}$ be  non-decreasing sequences of  real numbers 
greater than 1. Let $1\le p,q\le 2$ be such that $1/p+1/q=3/2$. 
There exists $C>0$ such that for any sequence $(\alpha_n)_{n\ge 1}$ and 
$(\beta_n)_{n\ge 1}$ of complex numbers such that 
\begin{equation}\label{cond-inemax-series-3}  \sum_{n\ge 1} \Big(\sum_{k\,:\, n\le \lambda_k< n+1}|\alpha_k|\,\Big)^p
<\infty \qq  \text{and} 
\qq \sum_{n\ge 1} \Big(\sum_{k\,:\, n\le \mu_k< n+1}\big|\beta_k|\,\Big)^q
<\infty \,  ,
\end{equation}
  we have 
\begin{equation}\label{inemax-series-3}
\Big\|\sup_{N\ge 1} \big|\sum_{n=1}^N \alpha_n D(\lambda_n t)\big|\, \Big\|_{\S^2}
\le C \Big(\sum_{n\ge 1} \big(\sum_{k\,:\, n\le \lambda_k< n+1}|\alpha_k|\big)^p \Big)^{1/p}\Big(\sum_{n\ge 1} \big(\sum_{k\,:\, n\le
 \mu_k< n+1}|\beta_k|\big)^q\Big)^{1/q } 
\end{equation}
where $D(t):= \sum_{n\ge 1}\beta_n {\rm e}^{i\mu_n t}$ is defined in $\S^2$. 
 Moreover, the series $\sum_{n\ge 1} 
\alpha_n D(\lambda_nt)$ converges in $\S^2$ and for $\lambda$-a.e. $t\in \R$.
\end{theorem}

Before doing the proof let us mention the following immediate corollaries. We first  apply Theorem \ref{interpolation} with the choice  $\m_n =\log n$, $n\ge 1$ and $\l_k=k$, $k\ge 1$.

\begin{cor} \label{inemax-cor1}   Assume that  $$   \sum_{k\ge 1} |\alpha_k|^p
<\infty \qq  \text{and} 
\qq \sum_{n\ge 1} \Big(\sum_{k\,:\, 2^n\le k< 2^{n+1}}\big|\beta_k|\,\Big)^q
<\infty \,  ,
$$
for some $1\le p,q\le 2$  such that $1/p+1/q=3/2$. Let  $D(t):= \sum_{n\ge 1}\beta_n n^{it}$. Then the series $\sum_{k\ge 1} 
\alpha_k D(kt)$ converges in $\S^2$ and for $\lambda$-a.e. $t\in \R$.
\end{cor}
 \begin{example} Let $1/2<\a\le 1$. Choose $1/\a<p\le 2$ and  $q=2p/(3p-2)$ ($1\le q<2$). Let  $D(t)= \sum_{n\ge 1}\beta_n n^{it}$ and assume that  
\begin{equation}\label{excond}
   \sum_{n\ge 1} \Big(\sum_{k\,:\, 2^n\le k< 2^{n+1}}\big|\beta_k|\,\Big)^q
<\infty \,  .
\end{equation}

 Then the series
\begin{equation}\label{exconv}
\sum_{k\ge 1} \frac{D(kt)}{k^\a}\end{equation}
converges almost everywhere. 
  This extends  to Dirichlet series Hartman and Wintner result \cite{HW} showing that the series 
$\Phi_\a(x) =\sum_{k=1}^\infty \frac{\psi(kx)}{k^\a}$
converges almost everywhere. Here $\psi(x)= x-[x] -1/2= \sum_{j=1}^\infty \frac{\sin 2\pi jx}{j}$, and  $[x]$ is the
integer part of $x$. That result is also a special case of \eqref{exconv}: take $\b_n=1/j$ if $n =2^j$, $j\ge 1$ and $\b_n=0$ elsewhere. 
\end{example}
\begin{remark} To our knowledge \cite{HW} contains, among other results on $\Phi_\a$, the first convergence result for the series of dilates $\sum_{k=1}^\infty \a_k\psi(kx)$.
 
\end{remark}

Then, we  apply Theorem \ref{interpolation} with the choice  $\m_n = n$, $n\ge 1$ and $\l_k=k$, $k\ge 1$.

\begin{cor} \label{inemax-cor2}   Assume that  $$   \sum_{k\ge 1} |\alpha_k|^p
<\infty \qq  \text{and} 
\qq \sum_{j\ge 1}|b_j|^q
<\infty \,  ,
$$
 for some $1\le p,q\le 2$  such that $1/p+1/q=3/2$. Let  $D(t)=\sum_{\ell\ge 1} b_\ell e^{i\ell t}$. Then the series $\sum_{k\ge 1} 
\alpha_k D(kt)$ converges in $\S^2$ and for $\lambda$-a.e. $t\in \R$.
  \end{cor}
  
\begin{remark}
 Suppose   that $b_j= \mathcal O (1/j^\a)$ for some $1/2<\a\le 1$. Assume that $$   \sum_{k\ge 1} |\alpha_k|^p
<\infty ,
$$
 for some $1\le p<2/(3-2\a)$. Then $\sum_{j\ge 1}|b_j|^q
<\infty$ for $q$ such that $1/p+1/q=3/2$ and we have $1\le p,q\le 2$. We deduce from Corollary \ref{inemax-cor2} that the series $\sum_{k\ge 1} 
\alpha_k D(kt)$ converges in $\S^2$ and  for $\lambda$-a.e. $t\in \R$.
When $1/2<\a< 1$, the nearly optimal sufficient condition $\sum_{k\ge 1} |c_k|^2\exp\big\{\frac{K(\log k)^{1-\a}}{(\log\log k)^\a}\big\}<\infty$ in which $K=K(\a)$ has been recently established in \cite[Theorem 2]{ABSW}. See also \cite[Theorem 3.1]{W} for  conditions of individual type, i.e. depending on the support of the coefficient sequence.
When $\a= 1$, the optimal sufficient coefficient condition, namely that $\sum_{k=1}^\infty |\a_k|^2 (\log\log k)^{2+\e}$ converges for some $\e>0$ suffices for  the convergence almost everywhere, has been recently obtained by Lewko and Radziwill  \cite[Corollary 3]{LR}.

These results are  clearly better. 
However, we note that our results are, even in the trigonometrical case, independent from these ones, and concern a larger class of trigonometrical series $D(t)$.
\end{remark}
 
\begin{proof}[Proof of Theorem \ref{interpolation}] Clearly, we only need to prove \eqref{inemax-series-3}. Let $(\alpha_n)_{n\ge 1}$ and $(\beta_n)_{n\ge 1}$ be 
in $\ell^1(\N)$, fixed for all the proof. Let  $D(t):=\sum_{n\ge 1}\beta_n {\rm e}^{i\mu_n t}$. 
It is enough to prove that for every $N\ge 1$, 
\begin{equation*}
\Big\|\sup_{m=1}^N \big|\sum_{n=1}^m \alpha_n D(\lambda_n t)\big|\, \Big\|_{\S^2}
\le C \Big(\sum_{n\ge 1} \big(\sum_{k\,:\, n\le \lambda_k< n+1}|\alpha_k|\big)^p \Big)^{1/p}\Big(\sum_{n\ge 1} \big(\sum_{k\,:\, n\le
 \mu_k< n+1}|\beta_k|\big)^q\Big)^{1/q }\, ,
\end{equation*}
for a constant $C>0$ not depending on $N$, $(\alpha_n)_{n\ge 1}$ and $(\beta_n)_{n\ge 1}$. 

\smallskip

We shall do that by interpolating \eqref{inemax-series-1} and 
\eqref{inemax-series-2}.

\smallskip

Define Banach spaces as follows 
\begin{gather*}
X_1:=\Big\{(a_n)_{n\ge 1}\in \C^\N \, :\, 
\|(a_n)_{n\ge 1}\|_{X_1}:=\sum_{n\ge 1} \sum_{k\,:\, n\le \lambda_k< n+1}|a_k|<\infty\Big\},\\
X_2:=\Big\{(a_n)_{n\ge 1}\in \C^\N \, :\, 
\|(a_n)_{n\ge 1}\|_{X_2}:=\Big(\sum_{n\ge 1} \big(\sum_{k\,:\, n\le \lambda_k< n+1}|a_k|\big)^2\Big)^{1/2}<\infty\Big\},\\
Y_1:=\Big\{(b_n)_{n\ge 1}\in \C^\N \, :\, 
\|(b_n)_{n\ge 1}\|_{Y_1}:=\sum_{n\ge 1} \sum_{k\,:\, n\le \mu_k< n+1}|b_k|<\infty\Big\},\\
Y_2:=\Big\{(b_n)_{n\ge 1}\in \C^\N \, :\, 
\|(b_n)_{n\ge 1}\|_{Y_1}:=\Big(\sum_{n\ge 1} \big(\sum_{k\,:\, n\le \mu_k< n+1}|b_k|\big)^2\Big)^{1/2}<\infty\Big\}\, .
\end{gather*}

For every $t\in \R$, let 
$$J(t):= \min\Big\{j\in \N\, :\, 1\le j\le N,\, 
|\sum_{n=1}^j \alpha_n D(\lambda_n t)|= \sup_{m=1}^N |\sum_{n=1}^m \alpha_n D(\lambda_n t)| \Big\}.$$ 

Define a linear operator $T$ on $(X_1+X_2)\times (Y_2+Y_1)$ by setting 
$$
T((a_n)_{n\ge 1},(b_n)_{n\ge 1}):= \sum_{k=1}^N {\bf 1}_{\{k\le J(t)\}} 
a_k \Big(\sum_{\ell\ge 1}b_{\ell} {\rm e}^{i\lambda_k\mu_\ell t}\Big)\, .
$$

\smallskip

By Propositions \ref{prop-series-1} and \ref{prop-series-2}, $T $ is continuous from $X_1\times Y_2$ to $\S^2$ and from $X_2\times Y_1$ to 
$\S^2$.

\smallskip

It follows from paragraph 10.1 of Calder\'on \cite{Calderon} that for 
every $s\in [0,1]$ there exists $C_s$ such that, with the notations of \cite{Calderon}
 $$
\| T((a_n)_{n\ge 1},(b_n)_{n\ge 1})\|_{\S^2}\le C_s 
\|(a_n)_{n\ge 1}\|_{[X_1,X_2]_s}\, \|(b_n)_{n\ge 1}\|_{[Y_2,Y_1]_s}\, ,
$$
where 
$$
\|(a_n)_{n\ge 1}\|_{[X_1,X_2]_s}= \inf\{\|f\|_{\F}\, :\, f\in \F,\, 
f(s)=(a_n)_{n\ge 1}\}\, ,
$$
and $\F$ is the Banach space of continuous functions   $f$ from 
$\{z\in \C \,:\,0\le   {\rm Re}\, z \le 1\}$ to 
$X_1+X_2$, analytic on $\{z\in \C \,:\,0<  {\rm Re}\, z < 1\}$ 
such that for every $t\in \R$, $f(it)\in X_1$ and $f(1+it)\in X_2$ with $\lim_{|t|\to+\infty} f(it)=\lim_{|t|\to+\infty} f(1+it)=
0$, endowed with the norm 
$$
\|f\|_\F:= \max\Big(\sup_{t\in \R}\|f(it)\|_{X_1}\,,\,\sup_{t\in \R}
\|f(1+it)\|_{X_2}\Big)\, .
$$
 The norm $\|(b_n)_{n\ge 1}\|_{[Y_2,Y_1]_s}$ is defined similarly.

\smallskip
We shall now give an upper bound for $\|(a_n)_{n\ge 1}\|_{[X_1,X_2]_s}$. 
By homogeneity, we may assume that 
$$
\sum_{n\ge 1}\Big(\sum_{n\le \lambda_k<n+1} |a_k|\Big)^{2/(2-s)}=1\, .
$$

\smallskip

Let $\varepsilon>0$. Define an element   $f_\varepsilon$ of 
$\F$ by setting for every $z\in \C$ such that $0\le  {\rm Re}\, z
\le 1$, $f(z)= (c_n(z))_{n\ge 1}$ where, for every $n\ge 1$ and every 
$k\ge 1$ such that $n\le \lambda_k <n+1$, 
$$
c_k(z)= {\rm e}^{\varepsilon (z^2-s^2)}a_k\Big(\sum_{n\le \lambda_\ell <n+1} 
|a_\ell|\Big)^{(2-z)/(2-s)-1}\, ,
$$
if $\sum_{n\le \lambda_\ell <n+1} |a_\ell| \neq 0$ and $c_k(z)=0$ otherwise.

\smallskip

The introduction of $\varepsilon$ here is a standard trick to ensure 
the assumptions $\lim_{|t|\to+\infty} f_\varepsilon(it)=\lim_{|t|\to+\infty} f_\varepsilon(1+it)=0$.

\smallskip

Notice that $f_\varepsilon(s)= (a_n)_{n\ge 1}$. For every $t\in \R$, 
\begin{gather*}
\|f_\varepsilon(it)\|_{X_1}\le \sum_{n\ge 1} \Big(\sum_{n\le \lambda_k<n+1} |a_k|\Big)^{2/(2-s)}=1\, .
\end{gather*}
Similarly, for every $t\in \R$, 
\begin{gather*}
\|f_\varepsilon(1+it)\|_{X_2}\le {\rm e}^\varepsilon \sum_{n\ge 1} (\sum_{n\le \lambda_k<n+1} |a_k|)^{2/(2-s)}={\rm e}^\varepsilon\, .
\end{gather*}

Letting $\varepsilon\to0$, we infer that 
$$
\|(a_n)_{n\ge 1}\|_{[X_1,X_2]_s}\le 1 =\Big(\sum_{n\ge 1}(\sum_{n\le \lambda_k<n+1} |a_k|)^{2/(2-s)}\Big)^{\frac{2-s}2}\, .
$$
 Similarly, one can prove that

$$
\|(b_n)_{n\ge 1}\|_{[X_1,X_2]_s}\le\Big( \sum_{n\ge 1}(\sum_{n\le \lambda_k<n+1} |b_k|)^{2/(1+s)}\Big)^{\frac{1+s}2}\, .
$$

Taking $s= 2(1-1/p)$, yields the desired result. \end{proof}


\section{A   necessary condition for convergence almost everywhere.}\label{s5} 
Hartman \cite{H} has proved the following result
\begin{theorem}\label{Hartman}  Assume that 
\begin{eqnarray}\label{hadamard}\frac{\l_k}{\l_{k-1}}\ge q>1, \qq\qq k\ge 1.
\end{eqnarray}
 Assume that the   series $\sum_{k=1}^\infty a_k  {\rm e}^{i\l_kt}$ converges for almost all real $t$. Then the series 
$\sum_{k=1}^\infty |a_k|^2$ converges.
\end{theorem} 
 The proof  is  similar to the one of Zygmund \cite[Proof of Lemma 6.5, Ch. V]{Z} (see also  p.\,120--122 of the 1935's Edition). 
\begin{remark} \label{kac} The converse of Theorem \ref{Hartman} is due to Kac \cite{Ka}. If  $\sum_{k=1}^\infty |a_k|^2$ converges, then the   series $\sum_{k=1}^\infty a_k  {\rm e}^{i\l_kt}$ with $(\l_k)_{k\ge 1}$ verifying \eqref{hadamard}, converges for almost all real $t$. Kac's proof is a modification of Marcinkiewicz's. See Remark \ref{marcin}. In place of Fejer's theorem, another summation method is used. See Theorem 13 and  pages 84--85 in \cite{Ti}, and Theorem 21 in \cite{HR}. \end{remark} \vskip 2 pt Theorem \ref{Hartman} can be extended in the following way. 
 \begin{theorem}\label{H}Let $\{\l_k, k\ge 1\}$ be a increasing sequence of positive reals satisfying the  following condition
\begin{eqnarray}\label{sumdiff}M:=\sum_{  k\not=\ell\,,\,  k'\not=\ell'\atop (k,\ell)\neq(k',\ell')  
  }
\big(1-|(\l_k-\l_\ell)-(\l_{k'}-\l_{\ell'}) |\big)_+^2 \, <\infty .
 \end{eqnarray}
Assume that 
\begin{eqnarray}\label{locaecv}\l \Big\{\sum_{k} a_k {\rm e}^{i\l_kt}\ {\rm converges}\Big\}>0.
 \end{eqnarray}  Then the series  $\sum_{k=1}^\infty |a_k|^2$
converges.\end{theorem}
 \begin{remark} By  considering  integers $k$ such that $n\le \l_k<n+1/2$, next those such that $n+1/2\le \l_k\le n+1$, we observe that condition \ref{sumdiff} implies that
 $$ \sup_{n} \#\big\{k: n\le \l_k<n+1\big\}<\infty.$$
  \end{remark}
 We give an application. Recall that  a Sidon sequence   is a set of integers with the property that the pairwise sums of elements are all distinct. 
 \vskip 3 pt 
  As a corollary we get
\begin{cor}\label{sidon} Let  $\{\l_k, k\ge 1\}$ be a Sidon sequence.  Assume that  \eqref{locaecv} is satisfied. Then the
series  $\sum_{k=1}^\infty |a_k|^2$ converges.
\end{cor} 

\begin{remark} In contrast with Hadamard gap sequences,   Sidon sequences may grow  at most polynomially. See \cite{R} where it is for instance proved
  that  the sequence $\{n^5 + [\xi n^4], n\ge n_0\}$    is for some real number $\xi\in [0,1]$  and $n_0$ large, a Sidon sequence.  
  \end{remark}
\begin{proof}[Proof of Corollary \ref{sidon}] 
Let $ (k,\ell)\neq(k',\ell') $ with  $k\not=\ell$ and $ 
k'\not=\ell'$. As the equation $ \l_k-\l_\ell = \l_{k'}-\l_{\ell'} $ means $ \l_k+\l_{\ell'}= \l_\ell + \l_{k'}   $,   the fact that $\{\l_k, k\ge 1\}$ is
a Sidon sequence implies that the only possible solutions are   $k=k'$, $\ell'=\ell$ or $k=\ell$, $\ell'=k'$. The last one is impossible  by assumption,
and the first would mean that  $ (k,\ell)=(k',\ell') $ which is   excluded. Consequently,
 $\l_k-\l_\ell \neq \l_{k'}-\l_{\ell'}$. Hence the sum in
\eqref{sumdiff} is always zero.
\end{proof} 
 \begin{remark}\label{rHartman} 
It follows from Hartman's proof  that under condition \eqref{hadamard}, the sequence of differences  $\l_k-\l_\ell$, $k\not= \ell$ is a finite union of subsequences such that the difference of  any two numbers of the same subsequence  exceeds $1$.  These subsequences  fulfill  assumption \eqref{sumdiff} of Theorem \ref{H}, and thus Theorem \ref{Hartman} follows from Theorem \ref{H}.
\end{remark}
Theorem \ref{H} is a consequence of the following general necessary condition for almost everywhere convergence of series of functions.
 \begin{theorem}\label{Ha}Let $(X,\mathcal B,\tau)$ be a probability space. Let $\{g_k, k\ge 1\}\subset L^4(\tau)$   be a  sequence of functions with $\|g_k\|_{2,\tau}= 1$,  $\|g_k\|_{4,\tau}\le K$ and satisfying the  following condition  
\begin{eqnarray}\label{sumdiff1}M:=\sum_{  k\not=\ell\,,\,  k'\not=\ell'\atop (k,\ell)\neq(k',\ell')  
  } 
\big|\langle g_k \overline{g_\ell} ,g_{k'} \overline{g_{\ell'}} \rangle_\tau \big|^2   \, <\infty .
 \end{eqnarray}

Assume that 
\begin{eqnarray}\label{locaecv1}\tau  \Big\{ \sum_{k} a_k g_k(t) \ {\rm converges}\Big\}>0.
 \end{eqnarray}  Then the series  $\sum_{k=1}^\infty |a_k|^2$
converges.\end{theorem}
\begin{proof}[Proof of Theorem \ref{Ha}] We use Hartman's method and the below   classical generalization of Bessel's inequality.  
\begin{lem} {\rm (Bellman-Boas' inequality)} \label{bb}Let  $x,y_1,\ldots, y_n$ be elements of an inner product space $(H,\langle .,.\rangle)$, then 
$$ \sum_{i=1}^n |\langle x,y_i\rangle |^2\le \|x\|^2\Big\{\max_{1\le i\le n}\|y_i\|^2+ \Big(\sum_{1\le i\not=j\le n} |\langle
y_i,y_j\rangle |^2\Big)^{1/2}\Big\}.$$
\end{lem}
See \cite{B} for instance. 
 As
$$\Big\{t:\sum_{k} a_k g_k(t)\ {\rm converges}\Big\}= \bigcap_{\e>0}\bigcup_{V}\bigcap_{u>v>V}\Big\{ t:\,\big|\sum_{k=v}^u a_k g_k(t)\big|\le \e\,\Big\},$$ by assumption it follows that for any $\e>0$, there exists an integer $V$ such that if $$A:=
\bigcap_{u>v>V}\Big\{ \,\big|\sum_{k=v}^u a_k g_k(t)\big|\le \e\,\Big\},$$
then 
 \begin{eqnarray}\label{A}\tau(A)>0
 .\end{eqnarray}
Assume the series  $\sum_{k\ge 1}  |a_k|^2$ is divergent. We are going to prove that this will contradict \eqref{A}. 

\vskip 2 pt By squaring out, 
  \begin{eqnarray}\label{squaring} \int_A\Big|\sum_{k=n}^m a_k g_k(t)\Big|^2\tau(\dd t)&=& \tau(A)\sum_{k=n}^m |a_k|^2 + \sum_{k,\ell=n\atop k\neq \ell}^m
a_k\overline{a}_\ell\int_A g_k(t)\overline{g_\ell}(t)\tau(\dd t).
\end{eqnarray}
 By using Cauchy-Schwarz's inequality, 
 \begin{eqnarray*}  \Big|  \sum_{k,\ell=n\atop k\neq \ell}^m
a_k\overline{a}_\ell\int_A g_k(t)\overline{g_\ell}(t)\tau(\dd t) \Big|& \le & \Big(  \sum_{k,\ell=n\atop k\neq \ell}^m
|a_k|^2|a_\ell|^2  \Big)^{1/2}  \Big(\sum_{k,\ell=n\atop k\neq \ell}^m
 \Big|\int_A g_k(t)\overline{g_\ell}(t)\tau(\dd t)\Big|^2 \Big)^{1/2}.\end{eqnarray*}
 Applying Lemma \ref{bb} to the system of vectors of  
$L^2_\tau (\R)$, $\chi(A)$,  $g_k(t)\overline{g_\ell}(t)$, $n\le k,\ell\le m$, 
gives in view of the assumption made,
\begin{eqnarray*} \sum_{k,\ell=n\atop k\neq \ell}^m
 \Big|\int_A g_k(t)\overline{g_\ell}(t)\tau(\dd t)\Big|^2&\le & \tau(A)^2  \Big\{ K^2+ \Big(\sum_{{{(k,\ell)\neq(k',\ell')
}\atop n\le k\neq\ell\le m}\atop n\le k'\neq\ell'\le m}
\big|\langle g_k \overline{g_\ell} ,g_{k'} \overline{g_{\ell'}} \rangle_\tau \big|^2\Big)^{1/2}   \Big\}
\cr &\le  & \tau(A)^2   \big\{  K^2+
M^{1/2}  
\big\}.\end{eqnarray*}
Letting $n,m$ tend to infinity, it follows that the series 
$\sum_{ k\neq \ell}  \big|\int_A g_k(t)\overline{g_\ell}(t)\tau(\dd t)\big|^2$   converges. Consequently,  for all $m>n$, $n$  sufficiently large ($n>N$, $N$ depending on $A$) we have
\begin{eqnarray*} \sum_{k,\ell=n\atop k\neq \ell}^m
  \Big|\int_A g_k(t)\overline{g_\ell}(t)\tau(\dd t)\Big|^2 &\le  & \tau(A)^2 /4 . \end{eqnarray*}There is no loss to assume $N>V$, which we do. Therefore
\begin{eqnarray*}  \Big|  \sum_{k,\ell=n\atop k\neq \ell}^m
a_k\overline{a}_\ell\int_A g_k(t)\overline{g_\ell}(t)\tau(\dd t) \Big|& \le &    \Big(  \sum_{k,\ell=n\atop k\neq \ell}^m
|a_k|^2|a_\ell|^2  \Big)^{1/2}\Big(\frac{\tau(A)}{2}\Big).\end{eqnarray*}
This along with \eqref{squaring} implies 
\begin{eqnarray}\label{squaring1} \int_A\Big|\sum_{k=n}^m a_kg_k(t)\Big|^2\tau(\dd t)&\ge & \Big(\frac{\tau(A)}{2}\Big)\sum_{k=n}^m |a_k|^2,\end{eqnarray}
for all $m>n>N$.
\vskip 2 pt  We get
\begin{eqnarray} \label{squaring2}\Big(\frac{\tau(A)}{2}\Big)\sum_{k=n}^m |a_k|^2\le   \int_A\Big|\sum_{k=n}^m a_k g_k(t)\Big|^2\tau(\dd t)\le \e^2  \tau(A)  ,
\end{eqnarray}
where for the last inequality we have used the fact that $N>V$ and the  definition of $A$.
\vskip 2 pt
 We are now free to let $m$ tend to infinity in \eqref{squaring2}, which we do. We deduce that  necessarily  $\tau(A)=0$. Hence a contradiction with
\eqref{A}.

This achieves the proof.
 \end{proof} 

\begin{proof}[Proof of Theorem \ref{H}] Choose $\tau(\dd t) $ as  the density  
function on the real line  associated to 
 $\tau (  t) =\frac{1-\cos t}{\pi  t^2}$. Then 
  $$\int_\R \tau (\dd t) =1, \qq \int_\R {\rm e}^{ixt}\tau (\dd t) = \big( 1- {|x|} \big)_+. $$

Since $\tau$ is absolutely continuous with respect to the Lebesgue measure,  \eqref{locaecv} holds with $\tau$ in place of $\l$. 
Next choose $g_k(t)= {\rm e}^{i\l_k t}$. We have 
$$\langle g_k \overline{g_\ell} ,g_{k'} \overline{g_{\ell'}} \rangle_\tau=  \big(1-|(\l_k-\l_\ell)-(\l_{k'}-\l_{\ell'}) |\big)_+  \, .
$$
Condition \eqref{sumdiff1} is thus fulfilled. Theorem \ref{Ha} applies and we deduce that the  series  $\sum_{k=1}^\infty |a_k|^2$
converges.\end{proof}

\vskip 5 pt  {\bf Final note.} While finishing this paper, we discovered that Theorem \ref{gen-prop} was proved by Guniya \cite{Guniya}  using a completely different  method from ours.
Guniya's proof makes use of Wiener's result \cite{Wi} (previously mentionned) and does not seem to provide directly a maximal inequality.  
Our proof is somewhat more elementary. Moreover  it allows to 
recover Wiener's result and provides at the same time a maximal  inequality.
     It seems that   Guniya's paper  has been completely overlooked among the mathematical community. We observe in particular that Theorem \ref{gen-prop} notably  includes     obviously
 Hedenmalm and  Saksman result \cite{HS} published nearly twenty years after \cite{Guniya}.
 \vskip 2pt We now briefly explain Guniya's approach (see Theorem 1.2, (8) and Lemmas after and  paragraph   2.10). 
The proof follows from the combination of several different results proved in the paper, and is based on Riemann theory of trigonometric series \cite[Ch.\,XVI-8]{Z}. 
Assume that the coefficients are positive. Then the series $\sum_{n} c_n e^{i\lambda_n x}$
 converges in  $\S^2$ to some $f$.
Let $I,J$ be two intervals  with $|I|<2\pi$,   $|J|=2\pi$ and $I\not\subseteq  J$. Let $F$ be represented by the term by term integrated Fourier series of $f$, and let $L$ be a bump function of class $C^5$ equal to $1$ on $I$ and to $0$ on $J\backslash I'$ where $I\subset I'\not\subseteq  J$. Then by a theorem due to Zygmund (see \cite[Theorem 9.19]{Z}), the partial sums of the Fourier series of $f$ are uniformly equiconvergent  on $I$ with the partial sum of a trigonometric series  $\sum_m a_m e^{imx}$. 
Next,  if  $FL$ admits a second order   derivative  in the sense of distributions, say $g$, then the above  trigonometric series is the one of $g$. And the a.e. convergence  on $I$ follows from Carleson's theorem. lt remains to prove that under condition \ref{Wiener-cond},  $F$ has indeed second order Schwarz derivatives,  controlled by the $L^2$ norm of $f$, which should follow from Theorem 2.2 in \cite{Guniya}. 
\vskip 17 pt \noi   {\bf Acknowlegements} 
\vskip 9 pt Part of that work has been carried out while the first author was
 member of the laboratory MICS from CentraleSupelec.  The authors are  grateful to    Anna Rozanova-Pierrat  and  Vladimir Fock for translating    the paper \cite{Guniya}. The second author is pleased to thank Michael Lacey for discussions on  Carleson's theorem for integrals  and its equivalence with Carleson's theorem for series (see Remark \ref{2.3}), and for a proof of this equivalence \cite{La}. This point is actually also used in \cite{KQ}, but the reference given (Berkson, Paluszynski and Weiss  paper) does not however contain  anything corresponding. See also Zygmund  \cite[Theorem 9.19]{Z}.
\vskip 1 pt
 

 \end{document}